\documentclass{amsart}
\usepackage{setspace}
\usepackage{a4}
\usepackage{amssymb,amsmath,amsthm,latexsym}
\usepackage{amsfonts}
\usepackage{amsfonts}
\usepackage{graphicx}
\usepackage{textcomp}
\usepackage{cite}
\usepackage{enumerate}
\usepackage[mathscr]{euscript}
\usepackage{mathtools}
\newtheorem{theorem}{Theorem}[section]

\newtheorem{definition}[theorem]{Definition}

\title{This is the title}
\usepackage{amssymb}
\usepackage{amssymb}
\usepackage{amssymb}
\usepackage{amssymb}
\usepackage{amsmath}
\usepackage{tikz}
\usepackage{hyperref}
\usepackage{enumerate}
\usepackage{mathtools}
\usepackage{amsmath}
\usepackage{tikz}
\usepackage{amssymb}
\usepackage{amsmath}
\usepackage{tikz}
\usepackage{hyperref}
\usepackage{amssymb}
\usepackage{amssymb}
\usepackage{amssymb}
\usepackage{amssymb}
\usepackage{amsmath}
\usepackage{tikz}
\usepackage{hyperref}
\usepackage{enumerate}
\usepackage{mathtools}
\usepackage{amsmath}
\usepackage{tikz}
\usepackage{graphicx}
\usepackage{textcomp}
\usetikzlibrary{cd}
\begin{document}
\begin{center}
{\bf{DILATIONS OF LINEAR MAPS ON VECTOR SPACES}}\\
K. Mahesh Krishna and P. Sam Johnson\\
Department of Mathematical and Computational Sciences\\ 
National Institute of Technology Karnataka (NITK), Surathkal\\
Mangaluru 575 025, India  \\
Emails: kmaheshak@gmail.com, sam@nitk.edu.in\\

Date: \today
\end{center}

\hrule
\vspace{0.5cm}
\textbf{Abstract}: We continue the study  dilation of linear maps on vector spaces introduced by Bhat, De, and Rakshit. This notion is a variant of vector space dilation introduced by Han, Larson, Liu, and Liu. We derive vector space versions of  Wold decomposition, Halmos dilation, N-dilation, inter-twining lifting theorem and a variant of Ando dilation. It is noted further that unlike a kind of uniqueness of Halmos dilation of strict contractions on Hilbert spaces, vector space version of Halmos dilation can not be characterized.




\section{Introduction}
    Using functional calculus and Weierstrass polynomial approximation theorem, Halmos in 1950 proved an interesting result that every contraction on a Hilbert space can be lifted to unitary. 
\begin{theorem}\cite{HALMOS} (Halmos dilation)
Let $\mathcal{H}$ be a Hilbert space and $T:\mathcal{H}\to \mathcal{H}$ be a contraction. Then the operator 
\begin{align*}
U\coloneqq \begin{pmatrix}
T & \sqrt{I-TT^*}   \\
\sqrt{I-T^*T} & -T^*   \\
\end{pmatrix}
\end{align*}is unitary on 	$\mathcal{H}\oplus \mathcal{H}$. In other words,
\begin{align*}
T=P_\mathcal{H}U_\mathcal{H},
\end{align*}
 where $P_\mathcal{H}:\mathcal{H}\oplus \mathcal{H}\to \mathcal{H}\oplus \mathcal{H}$ is the orthogonal projection onto $\mathcal{H}$.
\end{theorem}
Three years later, Sz. Nagy extended the result of Halmos which reads as follows.
\begin{theorem}\cite{NAGY}\label{NAGYTHEOREM} (Sz. Nagy dilation)
Let $\mathcal{H}$ be a Hilbert space and $T:\mathcal{H}\to \mathcal{H}$ be a contraction. Then there exists a Hilbert space $\mathcal{K}$	which contains $\mathcal{H}$ isometrically and a unitary $U:\mathcal{K}\to \mathcal{K}$ such that 
\begin{align*}
T^n=P_\mathcal{H}U_\mathcal{H}^n, \quad \forall n=1, 2,\dots, 
\end{align*}
where $P_\mathcal{H}:\mathcal{K}\to \mathcal{K}$  is the orthogonal projection onto $\mathcal{H}$.
\end{theorem}
Following the Theorem \ref{NAGYTHEOREM}, extension of contractions on Hilbert spaces became an active area of research, known as  dilation theory. Some standard references for this theory are \cite{NAGYFOIAS, LEVYSHALIT, ARVESON, ORRGUIDED}. This study of contractions boosted the study of  other classes of operators not only on Hilbert spaces, but also on Banach spaces \cite{FACKLER, STROESCU, AKCOGLU}. In a recent paper \cite{BHATDERAKSHITH}, Bhat, De, and Rakshit abstracted the key ingradients in Halmos and Sz. Nagy dilation theorem and set up a set theoretic version of dilation theory. 
\begin{definition}\cite{BHATDERAKSHITH}\label{BHATSET}
Let $A$ be a (non empty)	set and $h:A\to A$ be a map. An injective power dilation of $h$ is a quadruple $(B, i, v,p)$, where $B$ is a set, $i:A\to B$,  $v:B\to B$ are injective maps,  $p:B\to B$ is an idempotent map such that $p(B)=i(A)$ and 
\begin{align*}
i(h^n(a))=p(v^n(i(a))), \quad \forall a \in A, \forall n \in \mathbb{Z}_+.
\end{align*}
A dilation $(B, i, v,p)$ of $h$ is said to be minimal if 
\begin{align*}
B=\bigcup\limits_{n=0}^\infty v^n(i(A)).
\end{align*}
\end{definition}
It is a simple observation that for Hilbert spaces, every operator cannot be dilated to unitary operator. However, the following is a surprising result for sets derived in \cite{BHATDERAKSHITH}.
\begin{theorem}\cite{BHATDERAKSHITH}
	Every map $h:A\to A$ admits a minimal injective power dilation.
\end{theorem}
Bhat, De, and Rakshit suceeded in obtaining fundamental theorems of dilations such as Wold decomposition, Halmos dilation,  Sz. Nagy dilation, inter-twining lifting theorem, Sarason's lemma, Ando dilation and BCL (Berger, Coburn and Lebow) theorem. Definition \ref{BHATSET} allowed the authors of \cite{BHATDERAKSHITH} to introduce the dilation of linear maps on vector spaces and showed that every linear map admits a minimal injective power dilation (see Section \ref{SECTIONTWO} for definition).

In this paper, we give the abstract study of dilation initiated by Bhat, De, and Rakshit for vector spaces. We follow a similar development as done in \cite{BHATDERAKSHITH}. We first derive Wold decomposition, followed by Halmos dilation. After that we derive an N-dilation result which is motivated from the construction of Egervary. Followed by this, we derive inter-twining lifting theorem.  At present exact analogoue of Ando dilation is not known but  a variant of that is given.  
Before ending the introduction, we note that  there is another  vector space approach of dilation theory by Han, Larson, Liu, and Liu \cite{HANLARSONLIULIU} which is motivated from the Naimark dilation theorem \cite{CZAJA, NAIMARK1, NAIMARK2}, dilation theory of frames \cite{LARSONSZAFRANCISZEK, HANLARSONGROUP, HANLARSONOPERATOR, HAN2014, CASAZZAHANLARSON, KASHINKUKILOVA, HANJFA} and Stinespring dilation theorem \cite{STINESPRING}.

\section{Dilations of linear maps}\label{SECTIONTWO}
Let $\mathcal{H}$ be a Hilbert space. We recall that an operator  $T:\mathcal{H}\to \mathcal{H}$ is called  a shift if   $\cap _{n=0}^\infty T^n(\mathcal{H})=\{0\} $. Classical Wold decomposition is the following.	
\begin{theorem}\cite{NAGYFOIAS}
(Wold decomposition) Let $T$ be an isometry on a Hilbert space $\mathcal{H}$. Then  $\mathcal{H}$ decomposes uniquely as $\mathcal{H}=\mathcal{H}_u\oplus \mathcal{H}_s$, where $\mathcal{H}_u$ and $\mathcal{H}_s$ are   $T$-reducing subspaces of $\mathcal{H}$, $T_{|\mathcal{H}_u}:\mathcal{H}_u\to \mathcal{H}_u$ is a  unitary and $T_{|\mathcal{H}_s}:\mathcal{H}_s \to \mathcal{H}_s$ is a shift.
\end{theorem}
We note that the definition of shift of an operator does not use the Hilbert space structure. Thus it can be formulated for vector spaces without modifications.
\begin{definition}
Let $\mathcal{V}$ be a  vector space and $T:\mathcal{V}\to  \mathcal{V}$ be a linear map. The map $T$ is said to be a shift if $\cap _{n=0}^\infty T^n(\mathcal{V})=\{0\} $.	
\end{definition}
We now have the vector space version of Wold decomposition.
\begin{theorem}(Wold decomposition for vector spaces)
Let $T$ be an injective linear map on a vector  space $\mathcal{V}$. Then  $\mathcal{V}$ decomposes  as $\mathcal{V}=\mathcal{V}_b\oplus \mathcal{V}_s$, where $\mathcal{V}_b$ is a  $T$-invariant subspace of $\mathcal{V}$, $T_{|\mathcal{V}_b}:\mathcal{V}_b\to \mathcal{V}_b$ is a bijection and $T_{|\mathcal{V}_s}:\mathcal{V}_s \to \mathcal{V}$ is a shift.	
\end{theorem}
\begin{proof}
Define $\mathcal{V}_b\coloneqq \cap _{n=0}^\infty T^n(\mathcal{V})$ and let $\mathcal{V}_s$ be a vector space complement of $\mathcal{V}_b$ in $\mathcal{V}$. We clearly have 	$\mathcal{V}=\mathcal{V}_b\oplus \mathcal{V}_s$. Now $T(\mathcal{V}_b)=T (\cap _{n=0}^\infty T^n(\mathcal{V}))\subseteq \cap _{n=0}^\infty T^n(\mathcal{V})=\mathcal{V}_b$. Thus  $\mathcal{V}_b$ is a  $T$-invariant subspaces of $\mathcal{V}$. We now try to show  that $T_{|\mathcal{V}_b}$ is a bijection. Since $T$ is already injective, it suffices to show that $T_{|\mathcal{V}_b}$ is surjective. Let $y\in \mathcal{V}_b$. Then there exists a sequence $\{x_n\}_{n=1}^\infty$ in $\mathcal{V}$ such that $y=Tx_1=T^2x_2=T^3x_3=\cdots .$ Since $T$ is injective we then have $x_1=Tx_2=T^2x_2=\cdots $. Therefore  $y=Tx_1$ and $x_1\in \mathcal{V}_b$. Thus  $T_{|\mathcal{V}_b}$ is surjective. We are now left with proving that $T_{|\mathcal{V}_s}$ is a shift. Let $y\in \cap _{n=0}^\infty (T_{|\mathcal{V}_s})^n(\mathcal{V}_s)\subseteq (\cap _{n=0}^\infty T^n(\mathcal{V}))\cap \mathcal{V}_s= \mathcal{V}_b\cap \mathcal{V}_s$.  Hence $y=0$ which completes the proof. 
\end{proof}
Since vector space complements are not unique, note that, we do not have uniqueness in Wold decomposition for vector spaces. We now derive Halmos dilation for linear maps on vector spaces.
\begin{theorem}(Halmos dilation for vector spaces)
Let $\mathcal{V}$ be a vector space  and 	$T: \mathcal{V} \to \mathcal{V}$ be a linear map. Then the operator 
\begin{align*}
U\coloneqq \begin{pmatrix}
T & I   \\
I & 0  \\
\end{pmatrix}
\end{align*}
is invertible  on 	$\mathcal{V}\oplus \mathcal{V}$. In other words,
\begin{align*}
T=P_\mathcal{V}U_\mathcal{V},
\end{align*}
where $P_\mathcal{V}:\mathcal{V}\oplus \mathcal{V}\to \mathcal{V}\oplus \mathcal{V}$ is the first coordinate  projection onto $\mathcal{V}$.
\end{theorem}
\begin{proof}
It suffices to produce inverse map for $U$. A direct calculation says that 	
\begin{align*}
V\coloneqq \begin{pmatrix}
0 & I   \\
I & -T  \\
\end{pmatrix}
\end{align*}
is the inverse of $U$.
\end{proof}
In the sequel, any invertible operator of the form 
\begin{align*}
 \begin{pmatrix}
T & B   \\
C & D  \\
\end{pmatrix},
\end{align*}
 where  $B,C,D:\mathcal{V} \to \mathcal{V}$ are linear operators, will be called as a Halmos dilation of $T$.
Now we observe that Halmos dilation is not unique. Using the theory of block matrices \cite{LUSHIOU} we can produce a variety of Halmos dilations for a given operator. Following are some classes of Halmos dilations.
\begin{enumerate}[\upshape(i)]
	\item If $T: \mathcal{V} \to \mathcal{V}$ is an invertible linear map and the linear operators $B,C,D:\mathcal{V} \to \mathcal{V}$ are such that $D-CT^{-1}B$ is invertible, then 
	the operator 
	\begin{align*}
	U\coloneqq \begin{pmatrix}
	T & B  \\
	C & D  \\
	\end{pmatrix}
	\text{ is a Halmos dilation of $T$ on $\mathcal{V}\oplus \mathcal{V}$ whose inverse is }
	\end{align*}
	\begin{align*}
	 \begin{pmatrix}
	T^{-1} +T^{-1}B(D-CT^{-1}B)^{-1}& -T^{-1}B(D-CT^{-1}B)^{-1}   \\
	-(D-CT^{-1}B)^{-1}CT^{-1} & (D-CT^{-1}B)^{-1}  \\
	\end{pmatrix}.
	\end{align*}
	\item $D: \mathcal{V} \to \mathcal{V}$ is an invertible linear map and the linear operators $B,C:\mathcal{V} \to \mathcal{V}$ are such that $T-BD^{-1}C$ is invertible, then 
	the operator 
	\begin{align*}
	 \begin{pmatrix}
	T & B  \\
	C & D  \\
	\end{pmatrix}
	\text{ is a Halmos dilation of $T$ on $\mathcal{V}\oplus \mathcal{V}$ whose inverse is }
	\end{align*}
	\begin{align*}
	\begin{pmatrix}
	(T-BD^{-1}C)^{-1} & -(T-BD^{-1}C)^{-1}BD^{-1}   \\
	-D^{-1}C(T-BD^{-1}C)^{-1} & D^{-1}+D^{-1}C(T-BD^{-1}C)^{-1}BD^{-1}  \\
	\end{pmatrix}.
	\end{align*}
	\item $B: \mathcal{V} \to \mathcal{V}$ is an invertible linear map and the linear operators $C,D:\mathcal{V} \to \mathcal{V}$ are such that $C-DB^{-1}T$ is invertible, then 
	the operator 
	\begin{align*}
	\begin{pmatrix}
	T & B  \\
	C & D  \\
	\end{pmatrix}
	\text{ is a Halmos dilation of $T$ on $\mathcal{V}\oplus \mathcal{V}$ whose inverse is }
	\end{align*}
	\begin{align*}
\begin{pmatrix}
-(C-DB^{-1}T)^{-1}DB^{-1} &  (C-DB^{-1}T)^{-1}  \\
B^{-1}+B^{-1}T(C-DB^{-1}T)^{-1}DB^{-1} & -B^{-1}T(C-DB^{-1}T)^{-1} \\
\end{pmatrix}.
\end{align*}	
	\item $C: \mathcal{V} \to \mathcal{V}$ is an invertible linear map and the linear operators $B,D:\mathcal{V} \to \mathcal{V}$ are such that $B-TC^{-1}D$ is invertible, then 
	the operator 
	\begin{align*}
	\begin{pmatrix}
	T & B  \\
	C & D  \\
	\end{pmatrix}
	\text{ is a Halmos dilation of $T$ on $\mathcal{V}\oplus \mathcal{V}$ whose inverse is }
	\end{align*}
		\begin{align*}
	\begin{pmatrix}
-C^{-1}D	(B-TC^{-1}D)^{-1} &  C^{-1}+C^{-1}D(B-TC^{-1}D)^{-1}TC^{-1}  \\
	(B-TC^{-1}D)^{-1} &  -(B-TC^{-1}D)^{-1}TC^{-1}\\
	\end{pmatrix}.
	\end{align*}
\end{enumerate}
Recently, Bhat and Mukherjee \cite{BHATMUKHERJEE} proved that there is certain kind of uniqueness of Halmos dilation for strict contractions in Hilbert spaces. Result reads as follows.
\begin{theorem}\cite{BHATMUKHERJEE}\label{BM}
Let $\mathcal{H}$ be a finite dimensional Hilbert space and $T:\mathcal{H}\to \mathcal{H}$ be a strict contraction.	Then Halmos dilation of $T$ on $\mathcal{H}\oplus \mathcal{H}$ is unitarily equivalent to 
\begin{align*}
\begin{pmatrix}
T & -\sqrt{I-TT^*}W   \\
\sqrt{I-T^*T} & T^*W  \\
\end{pmatrix}, \quad \text{ for some unitary operator } W:\mathcal{H}\to \mathcal{H}.
\end{align*}
\end{theorem}
We next derive a negative result to Theorem \ref{BM} for Halmos dilation in vector spaces.
\begin{theorem}
Let $\mathcal{V}$ be a finite dimensional vector space and $T:\mathcal{V} \to \mathcal{V}$ be a linear operator. Then there are Halmos dilations of $T$ which are not similar. 	
\end{theorem}
\begin{proof}
Note that 	\begin{align*}
 \begin{pmatrix}
T & T-I   \\
T+I & T \\
\end{pmatrix}
\end{align*}
is an invertible operator and hence is a Halmos dilation of $T$. It is now enough to show that the matrices 
	\begin{align*}
\begin{pmatrix}
T & T-I   \\
T+I & T \\
\end{pmatrix} \quad \text{ and } \quad \begin{pmatrix}
T & I   \\
I & 0 \\
\end{pmatrix} 
\end{align*} 
are not similar. Since $\mathcal{V}$ is finite dimensional, we can use the property of trace map to conclude that these matrices are not similar.
\end{proof}
It was Egervary \cite{EGERVARY} who observed that Halmos dilation of contraction can be extended finitely so that power of dilation will be dilation of power of contraction. This can be formally stated as follows.
\begin{theorem}\cite{EGERVARY} \label{EGERVARY}(N-dilation)
Let $\mathcal{H}$ be a Hilbert space and $T:\mathcal{H}\to \mathcal{H}$ be a contraction. Let $N$ be a natural number. Then the operator 
\begin{align*}
U\coloneqq \begin{pmatrix}
T & 0& 0 & \cdots &0 & \sqrt{I-TT^*}   \\
\sqrt{I-T^*T} & 0& 0 & \cdots &0& -T^*   \\
0&I&0&\cdots &0& 0\\
0&0&I&\cdots &0 & 0\\
\vdots &\vdots &\vdots & & \vdots &\vdots \\
0&0&0&\cdots &0 & 0\\
0&0&0&\cdots &I & 0\\
\end{pmatrix}_{(N+1)\times (N+1)}
\end{align*}is unitary on 	$\oplus_{k=1}^{N+1} \mathcal{H}$ and 
\begin{align*}
T^k=P_\mathcal{H}U_\mathcal{H}^k,\quad \forall k=1, \dots, N,
\end{align*}
where $P_\mathcal{H}:\oplus_{k=1}^{N+1} \mathcal{H}\to \oplus_{k=1}^{N+1} \mathcal{H}$ is the orthogonal projection onto $\mathcal{H}$. 		
\end{theorem}
  We now derive vector space version of Theorem \ref{EGERVARY}.
  \begin{theorem}(N-dilation for vector spaces)\label{NDILATIONVECTOR}
Let $\mathcal{V}$ be a vector space  and 	$T: \mathcal{V} \to \mathcal{V}$ be a linear map. Let $N$ be a natural number. Then the operator 
 \begin{align*}
 U\coloneqq \begin{pmatrix}
 T & 0& 0 & \cdots &0 & I   \\
 I & 0& 0 & \cdots &0& 0   \\
 0&I&0&\cdots &0& 0\\
 0&0&I&\cdots &0 & 0\\
 \vdots &\vdots &\vdots & & \vdots &\vdots \\
 0&0&0&\cdots &0 & 0\\
 0&0&0&\cdots &I & 0\\
 \end{pmatrix}_{(N+1)\times (N+1)}
 \end{align*}is invertible  on 	$\oplus_{k=1}^{N+1} \mathcal{V}$ and 
 \begin{align}\label{FINITEDILATIONEQUATION}
 T^k=P_\mathcal{V}U_\mathcal{V}^k,\quad \forall k=1, \dots, N,
 \end{align}
 where $P_\mathcal{V}:\oplus_{k=1}^{N+1} \mathcal{V}\to \oplus_{k=1}^{N+1} \mathcal{V}$ is the first coordinate  projection onto $\mathcal{V}$.
  \end{theorem}
\begin{proof}
A direct calculation of power of $U$ gives Equation (\ref{FINITEDILATIONEQUATION}). To complete the proof, now we need show that $U$ is invertible. Define
 \begin{align*}
V\coloneqq \begin{pmatrix}
0 & I& 0& 0 &  \cdots &0 & 0   \\
0 & 0& I& 0 &  \cdots &0& 0   \\
0&0&0& I&\cdots &0& 0\\
0&0&0& 0&\cdots &0 & 0\\
\vdots &\vdots  &\vdots & \vdots& & \vdots &\vdots \\
0&0&0& 0&\cdots &0 & I\\
I&-T&0& 0&\cdots &0 & 0\\
\end{pmatrix}_{(N+1)\times (N+1)}
\end{align*}
Then $UV=VU=I$. Thus $V$ is the inverse of $U$.	
\end{proof}
It was Schaffer \cite{SCHAFFER} who gave a proof of Sz. Nagy dilation theorem using infinite matrices. We now obtain a similar  result for vector spaces. In the following theorem, $\oplus_{n=-\infty}^{\infty} \mathcal{V}$ is the vector space defined by 
\begin{align*}
\oplus_{n=-\infty}^{\infty} \mathcal{V}\coloneqq \{ \{x_n\}_{n=-\infty}^\infty, x_n \in \mathcal{V}, \forall n \in \mathbb{Z}, x_n\neq 0 
\text{ only for finitely many } n's\}
\end{align*}
with respect to  natural operations.
\begin{theorem}\label{SCHAFFERVECTOR}
Let $\mathcal{V}$ be a vector space  and 	$T: \mathcal{V} \to \mathcal{V}$ be a linear map.  Let $U\coloneqq(u_{n,m})_{-\infty \leq n,m\leq \infty}$ be the operator defined on 
$\oplus_{n=-\infty}^{\infty} \mathcal{V}$ given by  the infinite matrix defined as follows:
\begin{align*}
u_{0,0}\coloneqq T, \quad u_{n,n+1}\coloneqq I, \quad \forall n \in \mathbb{Z},  \quad u_{n,m}\coloneqq 0 \quad  \text{ otherwise},
\end{align*}
i.e, 
\begin{align*}
U=\begin{pmatrix}
 &\vdots &\vdots & \vdots & \vdots & \vdots & \\
\cdots & 0 & I& 0 & 0&  0& \cdots & \\
\cdots & 0 & 0& I & 0& 0&\cdots  & \\
\cdots & 0&0&\underline{T}&I& 0&\cdots&\\
\cdots & 0&0&0&0& I&\cdots &\\
\cdots & 0&0&0&0& 0&\cdots &\\
 & \vdots &\vdots &\vdots &\vdots  & \vdots & \\
\end{pmatrix}_{\infty\times \infty}
\end{align*}
where $T$ is in the $(0,0)$  position (which is underlined), is invertible  on 	$\oplus_{n=-\infty}^{\infty} \mathcal{V}$ and 
\begin{align}\label{INFINITEDILATIONEQUATION}
T^n=P_\mathcal{V}U_\mathcal{V}^n,\quad \forall n\in \mathbb{N},
\end{align}
where $P_\mathcal{V}:\oplus_{n=-\infty}^{\infty} \mathcal{V}\to \oplus_{n=-\infty}^{\infty} \mathcal{V}$ is the first coordinate  projection onto $\mathcal{V}$.
\end{theorem}
\begin{proof}
We  get Equation (\ref{INFINITEDILATIONEQUATION}) by calculation of powers of $U$. The matrix   $V\coloneqq(v_{n,m})_{-\infty \leq n,m\leq \infty}$ defined by  
\begin{align*}
v_{0,0}\coloneqq 0, \quad v_{1,-1}\coloneqq -T, \quad v_{n,n-1}\coloneqq I, \quad \forall n \in \mathbb{Z},  \quad v_{n,m}\coloneqq 0 \quad  \text{ otherwise},
\end{align*}
i.e, 
\begin{align*}
U=\begin{pmatrix}
&\vdots &\vdots & \vdots & \vdots & \vdots & \\
\cdots & I & 0& 0 & 0&  0& \cdots & \\
\cdots & 0 & I& \underline{0} & 0& 0&\cdots  & \\
\cdots & 0&-T&I&0& 0&\cdots&\\
\cdots & 0&0&0&I& 0&\cdots &\\
\cdots & 0&0&0&0& I&\cdots &\\
& \vdots &\vdots &\vdots &\vdots  & \vdots & \\
\end{pmatrix}_{\infty\times \infty}
\end{align*}
where $0$ is in the $(0.0)$  position (which is underlined), satisfies $UV=VU=I$ and hence $U$ is invertible which completes the proof.
\end{proof}
Note that the Equation (\ref{FINITEDILATIONEQUATION}) holds only upto $N$ and not for $N+1$ and higher natural numbers. An important observation associated with Theorems  \ref{NDILATIONVECTOR} and \ref{SCHAFFERVECTOR} is that the dilation is not optimal, i.e., even if the given operator is invertible, then also $U$ is not same as $T$. To overcome this, next we move on with the definition of dilation given by Bhat, De, and Rakshit \cite{BHATDERAKSHITH}.
\begin{definition}\cite{BHATDERAKSHITH}
Let $\mathcal{V}$ be a vector space  and 	$T: \mathcal{V} \to \mathcal{V}$ be a linear map. A linear injective dilation of $T$ is a quadruple $(\mathcal{W}, I, U,P)$, where $\mathcal{W}$ is a vector space,   and 	$I: \mathcal{V} \to \mathcal{W}$ is an  injective  linear map, $U: \mathcal{W} \to \mathcal{W}$ is an injective  linear map, $P: \mathcal{W} \to \mathcal{W}$ is an idempotent linear map such that $P(\mathcal{W})=I(\mathcal{W})
$ and 
\begin{align*}
\text{(Dilation equation)} \quad IT^nx=PU^nIx, \quad \forall n\in \mathbb{Z}_+,  \forall x \in  \mathcal{V}.
\end{align*}
A dilation $(\mathcal{W}, I, U,P)$ of $T$ is said to be minimal if 
\begin{align*}
 \mathcal{W}=\operatorname{span}\{U^nIx:  n\in \mathbb{Z}_+,   x \in  \mathcal{V}\}.
\end{align*}
\end{definition}
An easier way to remember the dilation equation is the following commutative diagram. 
\begin{center}
	\[
	\begin{tikzcd}
\mathcal{W}	 \arrow[r,"U^n"]&\mathcal{W}  \arrow[r,"P"]& \mathcal{W}\\
	&\mathcal{V} \arrow[ul,"I"] \arrow[r,"T^n"] & \mathcal{V}\arrow[u,"I"]
	\end{tikzcd}
	\]
\end{center}
In \cite{BHATDERAKSHITH} vector space analogous of Sz. Nagy dilation result was proved. 
\begin{theorem}\cite{BHATDERAKSHITH}\label{STANDARDDILATION}
	Every linear map $T: \mathcal{V} \to \mathcal{V}$ admits  minimal injective linear dilation.
\end{theorem}
\begin{proof}
	We reproduce the proof given by Bhat, De, and Rakshit \cite{BHATDERAKSHITH} for the sake of future use. Define 
	\begin{align*}
	\mathcal{W}\coloneqq\{(x_n)_{n=0}^\infty :x_n \in \mathcal{V}, \forall n \in  \mathbb{Z}_+, x_n\neq 0 \text{ only for finitely many } n's\}.
	\end{align*} 
	Clearly $\mathcal{W}$ is a vector space w.r.t. natural operations. Now define 
	\begin{align*}
&	I:\mathcal{V} \ni x \mapsto (x, 0, \dots ) \in  \mathcal{W},\\
&	U: \mathcal{W} \ni (x_n)_{n=0}^\infty \mapsto (0, x_0, \dots) \in \mathcal{W},\\
&	P:\mathcal{W} \ni (x_n)_{n=0}^\infty \mapsto \sum_{n=0}^{\infty}IT^nx_n\in \mathcal{W}.
	\end{align*}
	Then $(\mathcal{W}, I, U,P)$ is a minimal injective linear dilation of $T$.
\end{proof}
We call the dilation $(\mathcal{W}, I, U,P)$ constructed in Theorem \ref{STANDARDDILATION} as the standard dilation of $T$. We next consider inter-twining lifting theorem. For contractions acting on Hilbert spaces this says that any operator which intertwins contractions can be lifted so that the lifted operator intertwins dilation operator.
\begin{theorem}\cite{NAGYLIFTING} (Inter-twining lifting theorem) Let $T_1:\mathcal{H}_1\to \mathcal{H}_1$, $T_2:\mathcal{H}_2\to \mathcal{H}_2$  be contractions, where $\mathcal{H}_1$, $\mathcal{H}_2$ are Hilbert spaces. Let $V_1:\mathcal{K}_1\to \mathcal{K}_1$, $V_2:\mathcal{K}_2\to \mathcal{K}_2$ be minimal isometric dilations of $T_1,T_2$, respectively. Assume that $S:\mathcal{H}_2\to \mathcal{H}_1$ is a bounded linear operator such that $T_1S=ST_2$. Then there exists a bounded linear operator $R:\mathcal{K}_2\to \mathcal{K}_1$ such that $V_1R=RV_2$, $P_{\mathcal{H}_1}R_{\mathcal{H}_2^\perp}=0$, $P_{\mathcal{H}_1}R_{\mathcal{H}_2}=S$ and $\|R\|=\|S\|$. Conversely if $R:\mathcal{K}_2\to \mathcal{K}_1$ is a bounded linear operator such that $V_1R=RV_2$ and $P_{\mathcal{H}_1}R_{\mathcal{H}_2^\perp}=0$, then $S\coloneqq P_{\mathcal{H}_1}R_{\mathcal{H}_2}$ satisfies $T_1S=ST_2$. 
\end{theorem}

\begin{theorem}(Inter-twining lifting theorem for vector spaces) Let $\mathcal{V}_1$, $\mathcal{V}_2$ be vector spaces,   $T_1: \mathcal{V}_1 \to \mathcal{V}_1$, $T_2: \mathcal{V}_2 \to \mathcal{V}_2$ be linear maps. Let $(\mathcal{W}_1, I_1, U_1,P_1)$, $(\mathcal{W}_2, I_2, U_2,P_2)$ be standard dilations of   $T_1$, $T_2$, respectively. If  $S: \mathcal{V}_2 \to \mathcal{V}_1$ is a linear map such that $T_1S=ST_2$, then there exists a linear map  $R: \mathcal{W}_2 \to \mathcal{W}_1$ such that 
	\begin{align}\label{INTERIN}
	U_1R=RU_2, \quad RP_2=P_1R, \quad RI_2=I_1S.
	\end{align}
	Conversely if $R: \mathcal{W}_2 \to \mathcal{W}_1$ is a linear map such that $U_1R=RU_2,  RP_2=P_1R$, then there exists a linear map   $S: \mathcal{V}_2 \to \mathcal{V}_1$ such that 
	\begin{align}\label{INTERSECOND}
	RI_2=I_1S, \quad T_1S=ST_2.
	\end{align}
\end{theorem}
\begin{proof}
	Define $R:\mathcal{W}_2 \ni (x_n)_{n=0}^\infty \mapsto (Sx_n)_{n=0}^\infty \in \mathcal{W}_1 $. We now verify three equalities in Equation (\ref{INTERIN}). Let $ (x_n)_{n=0}^\infty \in \mathcal{W}_2$. Then 
	\begin{align*}
	&	U_1R(x_n)_{n=0}^\infty=	U_1(Sx_n)_{n=0}^\infty=(0, S x_0, Sx_1, \dots), \\ 
	&RU_2(x_n)_{n=0}^\infty=R(0, x_0, x_1,\dots)=(0, S x_0, Sx_1, \dots), 
		\end{align*}
	\begin{align*}	
	RP_2(x_n)_{n=0}^\infty&=R\left(\sum_{n=0}^{\infty}I_2T_2^nx_n\right)=\sum_{n=0}^{\infty}RI_2T_2^nx_n\\
	&=\sum_{n=0}^{\infty}R(T_2^nx_n, 0, 0, \dots)=\sum_{n=0}^{\infty}(ST_2^nx_n, 0, 0, \dots), 
	\end{align*}
	\begin{align*}
	 P_1R(x_n)_{n=0}^\infty&= P_1(Sx_n)_{n=0}^\infty=\sum_{n=0}^{\infty}I_1T_1^nSx_n\\
	&=\sum_{n=0}^{\infty}I_1ST_2^nx_n=\sum_{n=0}^{\infty}(ST_2^nx_n, 0, 0, \dots), 
		\end{align*}
	\begin{align*}
	& RI_2x=R(x, 0, 0, \dots)=(Sx, 0, 0, \dots), \quad I_1Sx=(Sx, 0, 0, \dots).
	\end{align*}
	We now consider the converse part. For this, first we have to define linear map $S$. Let $y \in \mathcal{V}_2$. Now $RP_2(y, 0, \dots)=P_1R(y, 0, \dots)\in I_1(\mathcal{V}_1)$ and  $I_1$ is injective implies that there exists a unique $x \in \mathcal{V}_2$ such that $RP_2(y, 0, \dots)=P_1R(y, 0, \dots)=I_1(x)$. We now define $Sy\coloneqq x.$ Then $S$ is well-defined and linear. Let $y \in \mathcal{V}_2$ and $x \in \mathcal{V}_2$ be such that $Sy=x$. Then $I_1Sy=RP_2(y, 0, \dots)=RI_2y$. Thus we verified first equality in  (\ref{INTERSECOND}). We are left with verification of second equality. We now calculate 
	\begin{align}\label{CONVERSE-1}
	RP_2U_2(x, 0, \dots)=RP_2(0, x, 0, \dots)=RI_2T_2x
	\end{align}
	and 
	\begin{align}\label{CONVERSEZERO}
	P_1U_1R(x, 0, \dots)&=P_1RU_2(x, 0, \dots)=P_1R(0,x,0 \dots)\\
	&=RP_2(0,x,0 \dots)=RI_2T_2x, \quad \forall x \in \mathcal{V}_2.
	\end{align}
	Given conditions produce
	\begin{align}\label{CONVERSE}
	RP_2U_2=P_1RU_2=P_1U_1R
	\end{align}
	Equation (\ref{CONVERSE}) says that (\ref{CONVERSE-1}) and (\ref{CONVERSEZERO}) are equal which completes the proof.
\end{proof}
Sz. Nagy's dilation theorem brings us to the question of dilating more than one operators which are commuting. After a decade of work of Sz. Nagy, Ando derived the following result.
\begin{theorem}\cite{ANDO} (Ando dilation)
Let $\mathcal{H}$ be a Hilbert space and $T_1, T_2:\mathcal{H}\to \mathcal{H}$ be commuting contractions.	Then there exists a Hilbert space $\mathcal{K}$	which contains $\mathcal{H}$ isometrically and a pair of commuting unitaries   $U_1, U_2:\mathcal{K}\to \mathcal{K}$ such that 
\begin{align*}
T_1^nT_2^m=P_\mathcal{H}U_1^n{U_2}_\mathcal{H}^m, \quad \forall n,m=1, 2, \dots, 
\end{align*}
where $P_\mathcal{H}:\mathcal{K}\to \mathcal{K}$  is the orthogonal projection onto $\mathcal{H}$.
\end{theorem}
It is known that Ando dilation theorem can not be extended for more than two commuting contractions \cite{BHATTACHARYYA}. However, it is a surprising result obtained by Bhat, De, and Rakshit \cite{BHATDERAKSHITH} that for set theoretic consideration, Ando dilation holds for arbitary number of functions.
We don't know Ando dilation for linear maps on vector spaces but have a variant of it which is given in the following theorem.
\begin{theorem}
Let $\mathcal{V}$ be a vector space  and 	$T, S: \mathcal{V} \to \mathcal{V}$ be commuting  linear maps.	Then there are dilations $(\mathcal{W}, I, U_1,P)$ and $(\mathcal{W}, I, U_2,P)$ of $T,S$ respectively, such that 
\begin{align*}
\begin{pmatrix}
0_c& U
\end{pmatrix}
=\begin{pmatrix}
0_r\\
V
\end{pmatrix}
\end{align*}
and 
\begin{align*}
 \quad IT^nS^mx=PU^nV^mIx, \quad \forall n,m\in \mathbb{Z}_+,  \forall x \in  \mathcal{V},
\end{align*}
where $0_c$ denotes the infinite column matrix of zero vectors and $0_r$ denotes the infinite row matrix of zero vectors.
\end{theorem}
\begin{proof}
We extend the construction in the proof of Theorem \ref{STANDARDDILATION}.	Define 
	\begin{align*}
	\mathcal{W}\coloneqq \bigg\{
	\begin{pmatrix}
	x_{0,0} & x_{0,1} & x_{0,2}&\cdots \\
	x_{1,0} & x_{1,1} & x_{1,2}&\cdots\\
	x_{2,0} & x_{2,1} & x_{2,2}&\cdots\\
	\vdots &\vdots &\vdots &\ddots 
	\end{pmatrix}_{\infty \times \infty }
	:x_{n,m} \in \mathcal{V}, \forall n,m \in  \mathbb{Z}_+, x_{n,m}\neq 0 \\
	\text{ only for finitely many } (n,m)'s\bigg\}.
	\end{align*}
Then $\mathcal{W}$ becomes a vector space with respect to natural operations.  We now define the following four linear maps:

	\begin{align*}
&	I:	\mathcal{V}\ni x \mapsto 	\begin{pmatrix}
	x & 0 & 0&\cdots \\
	0 & 0 & 0&\cdots\\
	0 & 0 & 0&\cdots\\
	\vdots &\vdots &\vdots &\ddots 
	\end{pmatrix} 
	\in \mathcal{W}\\
&	U: \mathcal{W} \ni \begin{pmatrix}
	x_{0,0} & x_{0,1} & x_{0,2}&\cdots \\
	x_{1,0} & x_{1,1} & x_{1,2}&\cdots\\
	x_{2,0} & x_{2,1} & x_{2,2}&\cdots\\
	\vdots &\vdots &\vdots &\ddots 
	\end{pmatrix} \mapsto \begin{pmatrix}
	0&0&0\\
	x_{0,0} & x_{0,1} & x_{0,2}&\cdots \\
	x_{1,0} & x_{1,1} & x_{1,2}&\cdots\\
\vdots &\vdots &\vdots &\ddots 
	\end{pmatrix} \in \mathcal{W}\\
&	V: \mathcal{W} \ni \begin{pmatrix}
	x_{0,0} & x_{0,1} & x_{0,2}&\cdots \\
	x_{1,0} & x_{1,1} & x_{1,2}&\cdots\\
	x_{2,0} & x_{2,1} & x_{2,2}&\cdots\\
	\vdots &\vdots &\vdots &\ddots 
	\end{pmatrix} \mapsto \begin{pmatrix}
	0&x_{0,0} & x_{0,1} &\cdots \\
0&	x_{1,0} & x_{1,1}&\cdots\\
0& 	x_{2,0} & x_{2,1} &\cdots \\
\vdots &	\vdots &\vdots  &\ddots 
	\end{pmatrix} \in \mathcal{W}\\
&	P :\mathcal{W} \ni \begin{pmatrix} x_{0,0} & x_{0,1} & x_{0,2}&\cdots \\
	x_{1,0} & x_{1,1} & x_{1,2}&\cdots\\
	x_{2,0} & x_{2,1} & x_{2,2}&\cdots\\
	\vdots &\vdots &\vdots &\ddots 
	\end{pmatrix} \mapsto \sum_{m=0}^{\infty}\sum_{n=0}^{\infty}IT^nS^mx_{n,m} \in  \mathcal{W}
	\end{align*}
	We then have 
	\begin{align*}
	\begin{pmatrix}
	0_c& U
	\end{pmatrix}=
	\begin{pmatrix}
0&	0 & 0 & 0&\cdots \\
0&	x_{0,0} & x_{0,1} & x_{0,2}&\cdots \\
0&	x_{1,0} & x_{1,1} & x_{1,2}&\cdots\\
0&	x_{2,0} & x_{2,1} & x_{2,2}&\cdots\\
0&	\vdots &\vdots &\vdots &\ddots 
	\end{pmatrix}
	=\begin{pmatrix}
	0_r\\
	V
	\end{pmatrix}.
	\end{align*}
Now $PU^nIx=IT^nx, $ $ 	PV^nIx=IS^mx$, $\forall x \in \mathcal{V}$, $\forall n,m\in \mathbb{Z}_+$. Hence  $(\mathcal{W}, I, U_1,P)$ and $(\mathcal{W}, I, U_2,P)$ are dilations of $T,S$, respectively. A calculation now shows that $ IT^nS^mx=PU^nV^mIx,  \forall n,m\in \mathbb{Z}_+,  \forall x \in  \mathcal{V}$.
\end{proof}

  \section{Acknowledgements}
First author thanks National Institute of Technology Karnataka (NITK) Surathkal for financial assistance. We  thank Sachin M. Naik for some discussions.  We are grateful to Prof. Orr Moshe Shalit, Technion, Israel for some arguments regarding his survey \cite{ORRGUIDED}. We thank the organizers (Dr. Sayan Chakraborty and Dr. Srijan Sarkar) of  WOTOA (Webinars on Operator Theory and Operator Algebras) for arranging a talk by  Prof. B. V. Rajarama Bhat.
  \bibliographystyle{plain}
 \bibliography{reference.bib}

\end{document}